\documentclass{amsart}

\usepackage[numbers, sort]{natbib}

\usepackage{amsmath}
\usepackage{amssymb}
\usepackage{amsmath}
\usepackage{amsfonts}
\usepackage{enumerate}
\usepackage{amsthm}
\usepackage{setspace}
\usepackage{multirow}
\usepackage{mathdots}
\usepackage{mathtools}
\usepackage{graphicx}
\usepackage{float}
\usepackage{subcaption}

\newtheorem{theorem}{Theorem}

\author{Xuefeng Shen}
\author{Melvin Leok}
\title{Geometric Exponential Integrators}

\begin{document}
\maketitle

\begin{abstract}
In this paper, we consider exponential integrators for semilinear Poisson systems. Two types of exponential integrators are constructed, one preserves the Poisson structure, and the other preserves energy. Numerical experiments for semilinear Possion systems obtained by semi-discretizing Hamiltonian PDEs are presented. These geometric exponential integrators exhibit better long time stability properties as compared to non-geometric integrators, and are computationally more efficient than traditional symplectic integrators and energy-preserving methods based on the discrete gradient method.
\end{abstract}

\section{Introduction}
Exponential integrators~\cite{HoOs2010} are a class of numerical integrators for stiff systems whose vector field can be decomposed into a linear term and a nonlinear term,
\begin{equation}\label{semilinear}
\dot{q} = Aq + f(q).
\end{equation}
Usually, the coefficient matrix $A$ has a large spectral radius, and is responsible for the stiffness of the system of differential equations, while the  nonlinear term $f(q)$ is relatively smooth. There are various ways to construct an exponential integrator~\cite{Mi2004}. For example, we can perform a change of variables $\tilde{q}(t) = e^{-At}q(t)$, and transform \eqref{semilinear} to obtain
\begin{equation}\label{ex}
  [e^{-At}q(t)]' = \tilde{q}^{'}(t) = e^{-At}f(e^{At}\tilde{q}(t)).
\end{equation}
Notice that the Jacobi matrix of \eqref{ex} equals $e^{-At}\nabla f e^{At}$, which has a smaller spectral radius than the Jacobi matrix $A + \nabla f$ of~\eqref{semilinear}.
A natural idea is to apply a classical integrator for the mollified system \eqref{ex} to obtain an approximation of $\tilde{q}(t)$, then invert the change of variables to obtain an approximation of the solution $q(t)$ of \eqref{semilinear}. In Section~\ref{symplecticex}, we shall demonstrate how to construct symplectic exponential integrators using this approach.

Another way of constructing exponential integrators starts from the variation-of-constants formula,
\begin{equation}\label{variationofconstant}
  q(t) = e^{A(t-t_0)}q(t_0) + \int_{t_0}^{t}e^{A(t-\tau)}f(q(\tau))d\tau,
\end{equation}
which is the exact solution for \eqref{semilinear} with initial condition $q(t_0) = q_0$. Then, a computable approximation can be obtained by approximating the $f(q(\tau))$ term inside the integral. If we approximate $f(q(\tau))$ by $f(q_k)$, we arrive at the exponential Euler method,
\begin{equation}\label{exeuler}
  q_{k+1} = e^{Ah}q_k + \int_{0}^{h}e^{A\tau}d\tau \cdot f(q_k).
\end{equation}
An exponential Runge--Kutta method of collocation type~\cite{HoOS2005} could also be constructed by approximating $f(q(\tau))$ with polynomials. In Section~\ref{senergyex}, we shall show how to construct energy-preserving exponential integrators from \eqref{variationofconstant}.

In this paper, we consider a specific form of \eqref{semilinear} which is a Poisson system. We assume $A = JD$, $f(q) = J\nabla V(q)$, where $J^\mathrm{T} = -J, D^\mathrm{T} = D$,
and $JD = DJ$, thus the coefficient matrix $A$ is also skew-symmetric. Now, the semilinear system \eqref{semilinear} can be written as,
\begin{equation}\label{poisson}
\dot{q} = J(Dq + \nabla V(q)) = J\nabla H(q),
\end{equation}
with Hamiltonian function $H(q) = \frac{1}{2}q^\mathrm{T}Dq + V(q)$. Equation \eqref{poisson} describes a constant Poisson system, and there are at least two quantities that are preserved by the flow:
the Poisson structure $J_{ij}\frac{\partial}{\partial x_i}\otimes \frac{\partial}{\partial x_j}$ and Hamiltonian $H(q)$. Geometric integrators that preserve the geometric structure and first integrals of the system typically exhibit superior qualitative properties when compared to non-geometric integrators, and they are an active area of research~\cite{HaLuWa2006, MaWe2001, LeRe2004}.

Here, we construct geometric exponential integrators that either preserve the Poisson structure or Hamiltonian of \eqref{poisson}. They exhibit long time stability, allow for relatively larger timesteps for the stiff problem, and are computationally more efficient as they can be implemented using fixed point iterations as opposed to Newton type iterations. For the rest of the paper, Section~\ref{symplecticex} is devoted to developing symplectic exponential integrators that preserve the Poisson structure; Section~\ref{senergyex} is devoted to developing energy preserving exponential integrators; numerical methods and experiments are presented in Section~\ref{numerical_methods} and Section~\ref{numerical_experiments}, respectively.

\section{Symplectic Exponential Integrator}\label{symplecticex}
For constant Poisson systems \eqref{poisson}, it was shown in~\cite{ZhQi1994} that the midpoint rule and diagonally implicit symplectic Runge--Kutta methods preserve the Poisson structure $J_{ij}\frac{\partial}{\partial x_i}\otimes \frac{\partial}{\partial x_j}$. We first start by constructing an exponential midpoint rule: apply the classical midpoint rule to the transformed system \eqref{ex} to obtain
\begin{equation}\label{eq1}
\frac{\tilde{q}_{k+1}-\tilde{q}_{k}}{h} = e^{-At_{k+1/2}}f\Big(e^{At_{k+1/2}}\frac{\tilde{q}_{k+1}+\tilde{q}_{k}}{2}\Big),
\end{equation}
where
\[ t_{k+1/2} = \frac{t_k+t_{k+1}}{2},\quad  h = t_{k+1}-t_k,\quad  \tilde{q}_{k} = e^{-At_k}q_k,\quad  \tilde{q}_{k+1} = e^{-At_{k+1}}q_{k+1}.
\]
Transform \eqref{eq1} back to $q_k$ and $q_{k+1}$, and we obtain the exponential midpoint rule,
\begin{equation}\label{exmid}
q_{k+1} = e^{Ah}q_k + h\cdot e^{A\frac{h}{2}}f\Big(\frac{e^{A\frac{h}{2}}q_k + e^{-A\frac{h}{2}}q_{k+1}}{2} \Big).
\end{equation}

\begin{theorem}\label{exmidtheorem}
 The exponential midpoint rule \eqref{exmid} preserves the Poisson structure when applied to the semilinear Poisson system \eqref{poisson}.
\end{theorem}

\begin{proof}
Recall that a map $\phi$ preserves Poisson structure $J_{ij}\frac{\partial}{\partial x_i}\otimes \frac{\partial}{\partial x_j}$ iff
\begin{equation}\label{eq2}
(\nabla \phi)J(\nabla \phi)^\mathrm{T} = J.
\end{equation}
Differentiating  \eqref{exmid}, we obtain,
\begin{gather*}
dq_{k+1} = e^{Ah}dq_k + h\cdot e^{A\frac{h}{2}}\nabla f \Big(\frac{1}{2}e^{A\frac{h}{2}}dq_k + \frac{1}{2}e^{-A\frac{h}{2}}dq_{k+1}\Big),\\
\Big(I - \frac{h}{2}\cdot e^{A\frac{h}{2}}\nabla f e^{-A\frac{h}{2}}\Big)dq_{k+1} = \Big(e^{Ah} + \frac{h}{2}\cdot e^{A\frac{h}{2}}\nabla f e^{A\frac{h}{2}}\Big)dq_k.
\end{gather*}
So the map $\phi(q_k)=q_{k+1}$ has Jacobi matrix $\nabla \phi = M^{-1}N$, where
\begin{align*}
	M &=  I - \frac{h}{2}\cdot e^{A\frac{h}{2}}\nabla f e^{-A\frac{h}{2}} =  I - \frac{h}{2}\cdot e^{A\frac{h}{2}}J \nabla^2 V e^{-A\frac{h}{2}},\\
	N &=  e^{Ah} + \frac{h}{2}\cdot e^{A\frac{h}{2}}\nabla f e^{A\frac{h}{2}} = e^{Ah} + \frac{h}{2}\cdot e^{A\frac{h}{2}}J\nabla^2 V e^{A\frac{h}{2}}.
\end{align*}
Then, we just need to verify \eqref{eq2}, which is equivalent to $MJM^\mathrm{T} = NJN^\mathrm{T}$,
\begin{equation}\label{eq3}
\begin{aligned}
        MJM^\mathrm{T} & = \Big(I - \frac{h}{2}\cdot e^{A\frac{h}{2}}J \nabla^2 V e^{-A\frac{h}{2}}\Big)J\Big(I + \frac{h}{2}\cdot e^{A\frac{h}{2}}\nabla^2 V Je^{-A\frac{h}{2}}\Big)\\
                       & = J - \frac{h^2}{4}e^{A\frac{h}{2}}J\nabla^2 V e^{-A\frac{h}{2}}Je^{A\frac{h}{2}}\nabla^2V Je^{-A\frac{h}{2}}\\
                       & = J - \frac{h^2}{4}e^{A\frac{h}{2}}J\nabla^2 V J\nabla^2V Je^{-A\frac{h}{2}}\\
                       & = NJN^\mathrm{T}.
\end{aligned}
\end{equation}
In \eqref{eq3}, we used the property that $\nabla^2V$ is symmetric, that $A$ is skew-symmetric which implies that $(e^{Ah})^\mathrm{T} = e^{-Ah}$, and the assumptions that $JD = DJ$ and that $e^{A\frac{h}{2}}$ and $J$ commute.
\end{proof}

The exponential midpoint rule is a second-order method, and we will now develop higher-order symplectic exponential integrators. Recall that a general diagonally implicit Symplectic Runge--Kutta method (DISRK) has a Butcher tableau of the following form,
\begin{table}[!hbp]
\centering
\caption{DISRK\label{tDISRK}}
\begin{tabular}{c|ccccc}
$c_1$        &  $\frac{b_1}{2}$      &   0                &          0            &      0         &        0                           \\
$c_2$        &  $b_1$                &  $\frac{b_2}{2}$   &             0         &        0       &         0                         \\
$c_3$        &  $b_1$                &  $b_2$             &    $\frac{b_3}{2}$    &        0       &           0                       \\
$\vdots$     & $\vdots$              &  $\vdots$          &    $\vdots$           &  $\ddots$      &           $\vdots$                     \\
$c_s$        & $b_1$                 &  $b_2$             &     $b_3$             &  $\cdots$      &    $\frac{b_s}{2}$              \\
\hline
             & $b_1$                 &  $b_2$             &     $b_3$             &  $\cdots$      &     $b_s$                         \\
\end{tabular}
\end{table}

It was shown in \cite{ZhQi1994} that any DISRK method can be regarded as the composition of midpoint rules with timesteps $b_1h, b_2h, b_3h, \cdots b_sh$. If we apply the DISRK method for the transformed system \eqref{ex}, and then convert back, we obtain the following diagonally implicit symplectic exponential (DISEX) integrator,
\begin{equation}\label{DIEX}
\left\{
\begin{aligned}
Q_i     &= e^{Ahc_i}q_k + h\cdot \displaystyle\sum_{j=1}^{i}a_{ij}e^{Ah(c_i-c_j)}f(Q_j),\\
q_{k+1} &= e^{Ah}q_k + h\cdot \displaystyle\sum_{i=1}^{s}b_i e^{Ah(1-c_i)}f(Q_i).
\end{aligned}\right.
\end{equation}
where $a_{ij}$ are the coefficients in Table~\ref{tDISRK}. This integrator can be represented in terms of the following Butcher tableau,
\begin{table}[!hbp]
\centering
\caption{DISEX\label{tDISEX}}
\begin{tabular}{c|ccccc}
$e^{Ahc_1}$        &  $\frac{b_1}{2}$                &   0                        &          0                 &      0         &        0                           \\
$e^{Ahc_2}$        &  $b_1e^{Ah(c_2-c_1)}$           &  $\frac{b_2}{2}$           &             0              &        0       &         0                         \\
$e^{Ahc_3}$        &  $b_1e^{Ah(c_3-c_1)}$           &  $b_2e^{Ah(c_3-c_2)}$      &    $\frac{b_3}{2}$         &        0       &           0                       \\
$\vdots$           & $\vdots$                        &  $\vdots$                  &    $\vdots$                &  $\ddots$      &           $\vdots$                     \\
$e^{Ahc_s}$        & $b_1e^{Ah(c_s-c_1)}$            &  $b_2e^{Ah(c_s-c_2)}$      &   $b_3e^{Ah(c_s-c_3)}$     &  $\cdots$      &    $\frac{b_s}{2}$              \\
\hline
 $e^{Ah}$          & $b_1e^{Ah(1-c_1)}$              &  $b_2e^{Ah(1-c_2)}$        &     $b_3e^{Ah(1-c_3)}$     &  $\cdots$      &     $b_se^{Ah(1-c_s)}$                         \\
\end{tabular}
\end{table}

The DISEX method preserves the Poisson structure, as demonstrated by the following theorem.
\begin{theorem}
The DISEX method is equivalent to the composition of exponential midpoint rules with timesteps $b_1h, b_2h, b_3h, \cdots b_sh$.
\end{theorem}
\begin{proof}
The composition of exponential midpoint rules $$q_k \xrightarrow{b_1h} Z_1 \xrightarrow{b_2h} Z_2 \cdots \xrightarrow{b_sh} Z_s = q_{k+1}$$
is represented as follows,
\begin{align}
Z_1 &= e^{Ahb_1}q_k + b_1h\cdot e^{Ah\frac{b_1}{2}}f\Big(\frac{e^{Ah\frac{b_1}{2}}q_k + e^{-Ah\frac{b_1}{2}}Z_1}{2}\Big), \label{M.1}\tag{M.1}\\
Z_2 &= e^{Ahb_2}Z_1+ b_2h\cdot e^{Ah\frac{b_2}{2}}f\Big(\frac{e^{Ah\frac{b_2}{2}}Z_1 + e^{-Ah\frac{b_2}{2}}Z_2}{2}\Big), \label{M.2}\tag{M.2}\\
&\qquad\qquad\qquad\qquad\vdots \notag\\
Z_i &= e^{Ahb_i}Z_{i-1} + b_ih\cdot e^{Ah\frac{b_i}{2}}f\Big(\frac{e^{Ah\frac{b_i}{2}}Z_{i-1} + e^{-Ah\frac{b_i}{2}}Z_i}{2}\Big), \label{M.i}\tag{M.i}\\
&\qquad\qquad\qquad\qquad\vdots \notag\\
Z_s &= e^{Ahb_s}Z_{s-1} + b_sh\cdot e^{Ah\frac{b_s}{2}}f\Big(\frac{e^{Ah\frac{b_s}{2}}Z_{s-1} + e^{-Ah\frac{b_s}{2}}Z_s}{2}\Big). \label{M.s}\tag{M.s}
\end{align}
Introduce $Q_1 = \frac{e^{Ah\frac{b_1}{2}}q_k + e^{-Ah\frac{b_1}{2}}Z_1}{2}$ in~\eqref{M.1} as an intermediate variable. Then, on both sides of~\eqref{M.1}, multiply by
$e^{-Ah\frac{b_1}{2}}$, add $e^{Ah\frac{b_1}{2}}q_k$, and divide by two, which yields an equivalent form of~\eqref{M.1},
\begin{equation}
Q_1 = e^{Ah\frac{b_1}{2}}q_k + \frac{b_1}{2}h\cdot f(Q_1). \label{S.1}\tag{S.1}
\end{equation}
For the Runge--Kutta method represented by the Butcher tableau in Table~\ref{tDISRK} to be consistent, the coefficients have to satisfy,
\begin{displaymath}
\left\{
\begin{aligned}
c_1 &=  \frac{b_1}{2}, \\
c_2 &=  b_1 + \frac{b_2}{2}, \\
    & \hspace{0.3em}\vdots \\
c_i &=  b_1 + b_2 +\cdots b_{i-1} + \frac{b_i}{2}, \\
    & \hspace{0.3em}\vdots \\
c_s &=  b_1 + b_2 +b_3 +\cdots +b_{s-1} + \frac{b_s}{2}, \\
1   &=  b_1 + b_2 +b_3 +\cdots +b_{s-1} + b_s.
\end{aligned} \right.
\end{displaymath}
So equation~\eqref{S.2} coincides with the first line of the Butcher tableau of the DISEX method. Similarly, introduce $Q_2 = \frac{e^{Ah\frac{b_2}{2}}Z_1 + e^{-Ah\frac{b_2}{2}}Z_2}{2}$. Then, on both
sides of~\eqref{M.2}, multiply by $e^{-Ah\frac{b_2}{2}}$, add $e^{Ah\frac{b_2}{2}}Z_1$, then divided by two, which yields an equivalent form of~\eqref{M.2}:
\begin{equation}
\begin{aligned}
Q_2 &= e^{Ah\frac{b_2}{2}}Z_1 + \frac{b_2}{2}h\cdot f(Q_2) \\
    &= e^{Ah(b_1 + \frac{b_2}{2})}q_k + b_1h\cdot e^{Ah(\frac{b_1}{2} + \frac{b_2}{2})}f(Q_1) + \frac{b_2}{2}h\cdot f(Q_2)\\
    &= e^{Ahc_2}q_k + b_1h\cdot e^{Ah(c_2 - c_1)}f(Q_1) + \frac{b_2}{2}h\cdot f(Q_2).
\end{aligned}\label{S.2}\tag{S.2}
\end{equation}
So equation~\eqref{S.2} coincides with the second line of the Butcher tableau of the DISEX method. Then, as before, we introduce $Q_i = \frac{e^{Ah\frac{b_i}{2}}Z_{i-1} + e^{-Ah\frac{b_i}{2}}Z_i}{2}$, and apply the same technique to~\eqref{M.i}, which yields
$$Q_i = e^{Ah\frac{b_i}{2}}Z_{i-1} + \frac{b_i}{2}h\cdot f(Q_i).$$
By induction,
\begin{align*}
Z_{i-1} &= e^{Ah(b_{i-1}+\cdots +b_2+b_1)}q_k + b_1h\cdot e^{Ah(b_{i-1}+b_{i-2}+\cdots +\frac{b_1}{2})}f(Q_1)\\
        &\qquad+ b_2h\cdot e^{Ah(b_{i-1}+b_{i-2}+\cdots +\frac{b_2}{2})}f(Q_2) + \cdots b_{i-1}h\cdot e^{Ah(\frac{b_{i-1}}{2})}f(Q_{i-1}),
\end{align*}
so
\begin{equation}
\begin{aligned}
Q_i &= ~e^{Ah(\frac{b_i}{2}+b_{i-1}+\cdots +b_2+b_1)}q_k + b_1h\cdot e^{Ah(\frac{b_i}{2}+ b_{i-1}+b_{i-2}+\cdots +\frac{b_1}{2})}f(Q_1) \\
    &\qquad+ b_2h\cdot e^{Ah(\frac{b_i}{2}+ b_{i-1}+b_{i-2}+\cdots +\frac{b_2}{2})}f(Q_2) + \cdots \\
    &\qquad+b_{i-1}h\cdot e^{Ah(\frac{b_i}{2}+ \frac{b_{i-1}}{2})}f(Q_{i-1})+ \frac{b_i}{2}h\cdot f(Q_i)\\
    &= e^{Ahc_i}q_k + b_1h\cdot e^{Ah(c_i - c_1)}f(Q_1) +\cdots\\
    &\qquad + b_{i-1}h\cdot e^{Ah(c_i - c_{i-1})}f(Q_{i-1}) + \frac{b_i}{2}h\cdot f(Q_i),
\end{aligned}    \tag{S.i}
\end{equation}
which coincides with the $i$-th row of the Butcher tableau of the DISEX method. Finally, we have
\begin{align*}
q_{k+1} &= Z_s \\
        &= e^{Ahb_s}Z_{s-1} + b_sh\cdot f(Q_s) \\
        &= e^{Ah(b_s+\cdots b_2 + b_1)}q_k + b_1h\cdot e^{Ah(b_s+\cdots b_2 + \frac{b_1}{2})}f(Q_1)+ \cdots \\
        &\qquad+ b_{s-1}h\cdot e^{Ah(b_s+\frac{b_{s-1}}{2})}f(Q_{s-1})+b_sh\cdot e^{Ah\frac{b_s}{2}}f(Q_s) \\
        &= e^{Ah}q_k + b_1h\cdot e^{Ah(1-c_1)}f(Q_1) + \cdots  b_sh\cdot e^{Ah(1-c_s)}f(Q_s),
\end{align*}
which coincides with the last row of the Butcher tableau of the DISEX method. So the composition of exponential midpoint rules with timesteps $b_1h, b_2h, b_3h, \cdots b_sh$ is equivalent to the DISEX method of Table~\ref{tDISEX}.
\end{proof}

\section{Energy-preserving Exponential Integrator}\label{senergyex}
Though classical symplectic methods exhibit superior long time stability, it was observed that symplectic schemes are less competitive for the numerical integration of stiff systems with high frequency. In sharp contrast, energy-preserving methods perform much better~\cite{SiGo1993}. A general way to construct an energy-preserving method for a Poisson system $\dot{q} = J\nabla H(q)$ is the discrete gradient method~\cite{QiTu1996}. We design a discrete gradient $\overline{\nabla} H(q_k, q_{k+1})$ that satisfies the following property,
\begin{equation}\label{dgcondition}
\overline{\nabla} H(q_k, q_{k+1})\cdot (q_{k+1} - q_k) = H(q_{k+1}) - H(q_k).
\end{equation}
Then, the resulting discrete gradient method is given by,
\begin{equation}\label{dg}
\frac{q_{k+1} - q_k}{h} = J\overline{\nabla} H(q_k, q_{k+1}).
\end{equation}
Multiplying $\overline{\nabla} H(q_k, q_{k+1})$ on both sides of \eqref{dg}, we obtain
\begin{equation}\label{eq44}
\begin{aligned}
        H(q_{k+1}) - H(q_k) & = \overline{\nabla} H(q_k, q_{k+1})\cdot (q_{k+1} - q_k)\\
                            & = h\cdot \overline{\nabla} H(q_k, q_{k+1})J\overline{\nabla} H(q_k, q_{k+1})\\
                            & = 0.
\end{aligned}
\end{equation}
The last equation of \eqref{eq44} holds simply due to the skew-symmetric property of matrix $J$, which implies that discrete gradient method~\eqref{dg} preserves energy. We shall combine exponential integrators with the discrete gradient method to obtain an energy-preserving exponential integrator. This approach was initially proposed in~\cite{WuLiSh2015} for separable Hamiltonian systems using the extended discrete gradient method, and we generalize this to semilinear Poisson systems. Replace the $f(q_k)$ term in the exponential Euler method \eqref{exeuler} by the discrete gradient $J\overline{\nabla} V(q_k, q_{k+1})$, which yields
\begin{equation}\label{energyex}
q_{k+1} = e^{Ah}q_k + \int_{0}^{h}e^{A\tau}d\tau \cdot J\overline{\nabla} V(q_k, q_{k+1}).
\end{equation}

\begin{theorem}
Method \eqref{energyex} preserves the Hamiltonian $H(q)$.
\end{theorem}
\begin{proof}
Let $S = e^{Ah}$, $T = \int_{0}^{h}e^{A\tau}d\tau$, then $q_{k+1} = Sq_k + TJ\overline{\nabla} V(q_k, q_{k+1})$, and we will show that the following properties hold:
\begin{enumerate}
  \item $S^\mathrm{T} = S^{-1}$;
  \item $AT = S - I$;
  \item $AT^\mathrm{T} = I - S^\mathrm{T}$;
  \item $S^\mathrm{T}T = T^\mathrm{T}$.
\end{enumerate}
Property 1 follows from the fact that $S^\mathrm{T} = (e^{Ah})^\mathrm{T} = e^{-Ah} = S^{-1}$. Property 2 follows from
$$e^{Ah} - I = e^{A\tau}\bigg|^h_0 = \int_{0}^{h}A\cdot e^{A\tau}d\tau = AT.$$
Taking transposes on both sides of Property 2 and using the fact that $A$ and $T$ commute gives Property 3. Property 4 follow from
$$S^\mathrm{T}T = e^{-Ah}\int_{0}^{h}e^{A\tau}d\tau = \int_{0}^{h}e^{-A(h-\tau)}d\tau = \int_{0}^{h}e^{-A\tau}d\tau = T^\mathrm{T}.$$

Recall that the Hamiltonian of the constant Poisson system \eqref{poisson} is $H(q) = \frac{1}{2}q^\mathrm{T}Dq + V(q)$, so
\begin{equation}\label{eq4}
\begin{aligned}
        H(q_{k+1}) &=\frac{1}{2}q_{k+1}^\mathrm{T}Dq_{k+1} + V(q_{k+1})\\
        &=\frac{1}{2}(Sq_k + TJ\overline{\nabla} V)^\mathrm{T}D(Sq_k + TJ\overline{\nabla} V) + V(q_{k+1})\\
        &=\frac{1}{2}q_k^\mathrm{T}S^\mathrm{T}DSq_k + q_k^\mathrm{T}S^\mathrm{T}DTJ\overline{\nabla} V \\
        &\qquad+\frac{1}{2}{\overline{\nabla} V}^\mathrm{T}J^\mathrm{T}T^\mathrm{T}DTJ\overline{\nabla} V + V(q_{k+1})\\
        &=\frac{1}{2}q_k^\mathrm{T}Dq_k + q_k^\mathrm{T}S^\mathrm{T}AT\overline{\nabla} V +\frac{1}{2}{\overline{\nabla} V}^\mathrm{T}J^\mathrm{T}T^\mathrm{T}AT\overline{\nabla} V + V(q_{k+1}),
\end{aligned}
\end{equation}
Applying the properties above yields the following,
\begin{gather*}
	q_k^\mathrm{T}S^\mathrm{T}AT\overline{\nabla} V = q_k^\mathrm{T}T^\mathrm{T}A\overline{\nabla} V = q_k^\mathrm{T}(I - S^\mathrm{T})\overline{\nabla} V,\\
	J^\mathrm{T}T^\mathrm{T}AT = J^\mathrm{T}T^\mathrm{T}(S - I)= J^\mathrm{T}T + JT^\mathrm{T},\\
	\frac{1}{2}{\overline{\nabla} V}^\mathrm{T}J^\mathrm{T}T^\mathrm{T}AT\overline{\nabla} V = \frac{1}{2}{\overline{\nabla} V}^\mathrm{T}(J^\mathrm{T}T + JT^\mathrm{T})\overline{\nabla} V = (\overline{\nabla} V)^\mathrm{T} T^\mathrm{T}J\overline{\nabla} V.
\end{gather*}
Substituting these into the expression for $H(q_{k+1})$ yields
\begin{equation}\label{eq4}
    \begin{aligned}
        H(q_{k+1}) & = \frac{1}{2}q_k^\mathrm{T}Dq_k + [(I - S)q_k]^\mathrm{T}\overline{\nabla} V - (\overline{\nabla} V)^\mathrm{T}(TJ)\overline{\nabla} V + V(q_{k+1})\\
        & = \frac{1}{2}q_k^\mathrm{T}Dq_k + (q_k - Sq_k -TJ\overline{\nabla} V)^\mathrm{T}\overline{\nabla} V + V(q_{k+1})\\
        & = \frac{1}{2}q_k^\mathrm{T}Dq_k + (q_k - q_{k+1})^\mathrm{T}\overline{\nabla} V(q_k,q_{k+1}) + V(q_{k+1})\\
        & = \frac{1}{2}q_k^\mathrm{T}Dq_k - V(q_{k+1}) + V(q_k) +V(q_{k+1})\\
        & = H(q_k),
    \end{aligned}
\end{equation}
which proves that the Hamiltonian is preserved.
\end{proof}
One significant advantage of exponential integrators is they allow the numerical method to be implemented using fixed point iterations as opposed to the more computationally expensive Newton iterations. Recall that classical implicit Runge-Kutta methods
\begin{displaymath}
\left\{
\begin{aligned}
    Y_i &= ~y_n + h\cdot \displaystyle\sum_{j=1}^{s}a_{ij}f(t_n + c_jh, Y_j), \\
y_{n+1} &= ~y_n + h\cdot \displaystyle\sum_{i=1}^{s}b_if(t_n + c_ih, Y_i),
\end{aligned} \right.
\end{displaymath}
have a form that naturally lends itself to fixed point iterations. However, when $\frac{\partial f}{\partial y}$ has a large spectral radius, the timestep is forced to be very small in order to guarantee that the fixed point iteration converges. The alternative is to use a Newton type iteration, which is time consuming since we need to perform LU decomposition ($O(n^3)$ complexity) during each iteration. This is the problem we face for the stiff semilinear system \eqref{semilinear} when the coefficient matrix $A$ has a large spectral radius. In contrast, in both the exponential midpoint rule
\begin{equation*}
q_{k+1} = e^{Ah}q_k + h\cdot e^{A\frac{h}{2}}f\Big(\frac{e^{A\frac{h}{2}}q_k + e^{-A\frac{h}{2}}q_{k+1}}{2}\Big),
\end{equation*}
and the energy-preserving exponential integrator
\begin{equation*}
q_{k+1} = e^{Ah}q_k + \int_{0}^{h}e^{A\tau}d\tau \cdot J\overline{\nabla} V(q_k, q_{k+1}),
\end{equation*}
the matrix $A$ only appears in the exponential term $e^{Ah}$. Since $A$ is skew-symmetric, this term is an orthogonal matrix, which has spectral radius 1. Thus, fixed point iterations can be used to implement the exponential integrator, regardless of the stiffness of $A$.

\section{Numerical Methods}\label{numerical_methods}
We consider two Hamilton PDEs here, namely the nonlinear Schr\"{o}dinger equation,
\begin{equation}\label{nls}
i{\psi}_t + {\psi}_{xx} - 2|\psi|^2\psi = 0,
\end{equation}
in which $\psi=u+iv$ is the wave function with real part $u$ and imaginary part $v$, and has the following equivalent form,
\begin{equation}\label{equNLS}
\left\{ \begin{aligned}
u_t &= -v_{xx} + 2(u^2+v^2)v,\\
v_t &= u_{xx} - 2(u^2+v^2)u;
\end{aligned} \right.
\end{equation}
as well as the KdV equation,
\begin{equation}\label{kdv}
u_t + uu_x + u_{xxx} = 0.
\end{equation}
To discretize the two PDEs,  we impose $2\pi$ periodic boundary conditions. Given a smooth $2\pi$ periodic function $f(x)$, on the interval $[0,2\pi]$, choose $2n+1$ equispaced interpolation points $x_j = jh, j=0,1,2,\ldots,2n$, $h=\frac{2\pi}{2n+1}$. Given nodal values $\{v_j\}_{j=0}^{2n}$, there exists a unique trigonometric polynomial $v(x)$ with degree less or equal $n$, such that, $v(x_j)=v_j$ (see, for example,~\cite{AtHa2009}).
\begin{displaymath}
           v(x)= \displaystyle\sum_{k=-n}^{n}\hat{v}_ke^{ikx}, \quad \hat{v}_k =\frac{1}{2n+1}\displaystyle\sum_{j=0}^{2n}v_je^{-ikx_j}.
\end{displaymath}
By substituting the expression for the coefficients $\hat{v}_k$, we obtain
\begin{equation}
    \begin{aligned}
        v(x) & = \displaystyle\sum_{k=-n}^{n}\Bigg(\frac{1}{2n+1}\displaystyle\sum_{j=0}^{2n}v_je^{-ikx_j}\Bigg)e^{ikx}\\
             & = \displaystyle\sum_{j=0}^{2n}v_j\Bigg(\displaystyle\sum_{k=-n}^{n}\frac{1}{2n+1}e^{ik(x-x_j)}\Bigg)\\
             & = \displaystyle\sum_{j=0}^{2n}v_j\phi(x-x_j)\\
             & = \displaystyle\sum_{j=0}^{2n}v_j\phi_j(x),
    \end{aligned}
\end{equation}
where
\begin{displaymath}
  \phi(x) = \displaystyle\sum_{k=-n}^{n}\frac{1}{2n+1}e^{ikx} = \frac{1}{2n+1}\frac{\sin((n+\frac{1}{2})x)}{\sin(\frac{x}{2})}.
\end{displaymath}
From this, we see that $\{e^{ikx}\}^n_{k=-n}$ and $\{\phi_j\}^{2n}_{j=0}$ are equivalent orthogonal bases for the trigonometric polynomial function space, and each such function can be parametrized by either the nodal values $\{v_j\}^{2n}_{j=0}$ or Fourier coefficients $\{\hat{v}_k\}^n_{k=-n}$. They represent the same function, but with respect to two different bases. The transformation between $\{v_j\}^{2n}_{j=0}$ and $\{\hat{v}_k\}^n_{k=-n}$ can be performed using the Fast Fourier transformation (FFT), which has $O(n\log n)$ complexity.
\qquad

The first and second-order differentiation matrices~\cite{Tr2000} with respect to the representation in terms of nodal values $\{v_j\}^{2n}_{j=0}$ are given by
\begin{align*}
(D_1)_{kj}&=\begin{cases}
           0, &k=j,\\
           \frac{(-1)^{(k-j)}}{2\sin (\frac{(k-j)h}{2})}, &k\ne j,
\end{cases}\\
(D_2)_{kj}&=\begin{cases}
           -\frac{n(n+1)}{3}, &k=j,\\
           \frac{(-1)^{(k-j+1)}\cos (\frac{(k-j)h}{2})}{2\sin^2 (\frac{(k-j)h}{2})}, &k\ne j,\\
\end{cases}
\end{align*}
respectively. However, with respect to the representation in terms of Fourier coefficients $\{\hat{v}_k\}^n_{k=-n}$, they are diagonal,
$$\hat{D}_1 = \operatorname{diag}(ik)^n_{k=-n}, \quad \hat{D}_2 = \operatorname{diag}(-k^2)^n_{k=-n}.$$
Later, we will see that this observation is critical to a fast implementation of the product of matrix functions with vectors. We can also define a third-order differentiation matrix $D_3$, which has the property $D_3 = D_1D_2 = D_2D_1$, and it is diagonal with respect to the Fourier coefficients $\hat{D}_3 = \operatorname{diag}(-ik^3)^n_{k=-n}$.

\subsection{Nonlinear Schr\"{o}dinger equation}
We perform a semi-discretization of \eqref{equNLS} by discretizing the solution $u$, $v$ in space using their corresponding nodal values $\{q_j\}$ and $\{p_j\}$. Applying the pseudospectral method, we obtain the following system of ODEs,
\begin{equation}\label{nlsode}
\left\{ \begin{aligned}
\dot{q} & = -D_2p + 2(q^2 + p^2)p, \\
\dot{p} & =  D_2q  - 2(q^2 + p^2)q,
\end{aligned} \right.
\end{equation}
where the nonlinear term $(q^2 + p^2)p$ is computed elementwise, and represents the vector consisting of $\{(q^2_j + p^2_j)p_j\}$ entries. We adopt this notation throughout the rest of the paper for brevity. Then, \eqref{nlsode} can be expressed as,
\begin{equation} \label{nlsodepoisson}
\frac{d}{dt}
  \begin{pmatrix}
   q \\
   p
  \end{pmatrix}
=\begin{pmatrix}
0 & -D_2 \\
D_2 & 0
\end{pmatrix}
\begin{pmatrix}
   q \\ƒ 
   p ƒ 
  \end{pmatrix}
   +
\begin{pmatrix}
   2(q^2 + p^2)p \\
   -2(q^2 + p^2)q
  \end{pmatrix},
\end{equation}
where
\begin{displaymath}
A = \begin{pmatrix}
0 & -D_2 \\
D_2 & 0
\end{pmatrix}
=
\begin{pmatrix}
0 & I \\
-I & 0	
\end{pmatrix}
\begin{pmatrix}
-D_2 & 0 \\
0 & -D_2
\end{pmatrix} =
J\cdot D,
\end{displaymath}

\begin{displaymath}
f(q, p ) = \begin{pmatrix}
   2(q^2 + p^2)p \\
   -2(q^2 + p^2)q
\end{pmatrix}
 =
\begin{pmatrix}	
0 & I \\
-I & 0
\end{pmatrix}
\begin{pmatrix}
   2(q^2 + p^2)q \\
   2(q^2 + p^2)p	
\end{pmatrix}
 =
J\cdot \overline{\nabla} V(q,p),
\end{displaymath}
and $V(q, p) = \frac{1}{2}(q^2 + p^2)^2$. It is easy to verify that $J$ is skew-symmetric, $D$ is symmetric, and $JD = DJ$. Thus, \eqref{nlsodepoisson} is a semilinear Poisson system.

To apply the exponential midpoint rule, we need to compute the product of a matrix function and a vector, which has the form $e^{Ah}\begin{pmatrix}q\\p\end{pmatrix}$,

\begin{align*}
  e^{\begin{psmallmatrix}
   0 & -D_2 \\
   D_2 & 0
 \end{psmallmatrix}} &=  \displaystyle\sum_{k=0}^{\infty}\frac{1}{k!}{\begin{psmallmatrix}
   0 & -D_2 \\
   D_2 & 0
 \end{psmallmatrix}}^k,  \\
&=  \displaystyle\sum_{k=0}^{\infty}\frac{1}{(2k)!}
\begin{psmallmatrix}
(-1)^k\cdot (D_2)^{2k} & 0 \\
0 & (-1)^k\cdot (D_2)^{2k}\\
\end{psmallmatrix}\\
&\qquad + \displaystyle\sum_{k=0}^{\infty}\frac{1}{(2k+1)!}
\begin{psmallmatrix}
0 & (-1)^{k+1}\cdot (D_2)^{2k+1} \\
(-1)^k \cdot (D_2)^{2k+1}& 0 \\
\end{psmallmatrix} \\
&=  \begin{psmallmatrix}
             \cos(D_2) & -\sin(D_2) \\
             \sin(D_2) & \cos(D_2)
           \end{psmallmatrix}.
\end{align*}

\begin{theorem} \label{matrixtimesv}
Suppose that $D_2$, $D_3$ are the second and third-order differentiation matrices, respectively, $q = \{q_j\}_{j=0}^{2n}$ is a vector with Fourier
transform $F[q] = \hat{q} = \{\hat{q}_k\}_{k=-n}^{n}$, and $f$ is an analytic function, then
$$f(D_2)q = F^{-1}[\operatorname{diag}(f(-k^2))\hat{q}],\qquad f(D_3)q = F^{-1}[\operatorname{diag}(f(-ik^3))\hat{q}],$$
where $F^{-1}$ is the inverse Fourier transform.
\end{theorem}
\begin{proof}
Recall that matrix $D_2$ is diagonalizable with eigenvalues $\lambda_k = -k^2$, and corresponding eigenvectors $e_k = \{e^{ikx_j}\}_{j=0}^{2n}$,
\begin{align*}
  f(D_2)q & = f(D_2)\Bigg(\displaystyle\sum_{k=-n}^{n}\hat{q}_k\cdot e_k\Bigg) \\
          & = \displaystyle\sum_{k=-n}^{n}\hat{q}_k \cdot f(D_2)e_k \\
          & = \displaystyle\sum_{k=-n}^{n}\hat{q}_k \cdot f(\lambda_k)e_k \\
          & = F^{-1}[\operatorname{diag}(f(\lambda_k))\hat{q}]\\
          & = F^{-1}[\operatorname{diag}(f(-k^2))\hat{q}].
\end{align*}
Notice $D_3$ is also diagonalizable with eigenvalues $\lambda_k = -ik^3$, and corresponding eigenvectors $e_k$, so the property that $f(D_3)q = F^{-1}[\operatorname{diag}(f(-ik^3))\hat{q}]$ can be verified in the same way.
\end{proof}

By Theorem~\ref{matrixtimesv}, we have that
\begin{equation}\label{eq100}
\begin{aligned}
  e^{\begin{psmallmatrix}
   0 & -D_2 \\
   D_2 & 0
 \end{psmallmatrix}h}
 \begin{pmatrix}
 	q \\
 	p
 \end{pmatrix}
 & =  \begin{pmatrix}
 		\cos(D_2h) & -\sin(D_2h) \\
    	\sin(D_2h) & \cos(D_2h)
	\end{pmatrix}
 \begin{pmatrix}
 	q \\
 	p
 \end{pmatrix}\\
  & =  \begin{pmatrix}\cos(D_2h)\cdot q - \sin(D_2h)\cdot p \\ \sin(D_2h)\cdot q + \cos(D_2h)\cdot p\end{pmatrix}\\
  & =  \begin{pmatrix} F^{-1}[\cos(k^2h)\hat{q}_k + \sin(k^2h)\hat{p}_k] \\ F^{-1}[\cos(k^2h)\hat{p}_k - \sin(k^2h)\hat{q}_k]\end{pmatrix}.
\end{aligned}
\end{equation}

In summary, the exponential midpoint rule for the nonlinear Schr\"{o}dinger equation is given by
\begin{equation*}
z_{k+1} = e^{Ah}z_k + h\cdot e^{A\frac{h}{2}}f\bigg(\frac{e^{A\frac{h}{2}}z_k + e^{-A\frac{h}{2}}z_{k+1}}{2}\bigg),
\end{equation*}
where $z_k =\begin{pmatrix}
   q_k \\
   p_k\end{pmatrix}$, $z_{k+1} = \begin{pmatrix}
   q_{k+1} \\
   p_{k+1}\end{pmatrix}$, $A = \begin{pmatrix}
0 & -D_2 \\
D_2 & 0\end{pmatrix}$, and  $e^{Ah}z_k$ can be efficiently calculated using \eqref{eq100}.

For the energy-preserving exponential integrator for the nonlinear Schr\"{o}dinger equation,
\begin{equation*}
z_{k+1} = e^{Ah}z_k + \int_{0}^{h}e^{A\tau}d\tau \cdot J\overline{\nabla} V(z_k, z_{k+1}).
\end{equation*}
Here,
\begin{align*}
  \int_{0}^{h} e^{\begin{psmallmatrix}
   0 & -D_2 \\
   D_2 & 0
 \end{psmallmatrix}\tau} d\tau &=  \int_{0}^{h}\begin{psmallmatrix}
             \cos(D_2\tau) & -\sin(D_2\tau) \\
             \sin(D_2\tau) & \cos(D_2\tau)
           \end{psmallmatrix}d\tau \\
  &=  h\cdot \begin{pmatrix}
               \frac{\sin(D_2h)}{D_2h} & \frac{\cos(D_2h)-1}{D_2h} \\
               \frac{1-\cos(D_2h)}{D_2h} & \frac{\sin(D_2h)}{D_2h}
             \end{pmatrix},
\end{align*}
and we can construct the discrete gradient
\begin{displaymath}
\overline{\nabla} V(z_k,z_{k+1}) = \begin{pmatrix}
                                2({q^2}_{k+\frac{1}{2}}+{p^2}_{k+\frac{1}{2}})\cdot q_{k+\frac{1}{2}} \\
                                2({q^2}_{k+\frac{1}{2}}+{p^2}_{k+\frac{1}{2}})\cdot p_{k+\frac{1}{2}}
                              \end{pmatrix},
\end{displaymath}
where
\begin{alignat}{2}
q_{k+\frac{1}{2}} &= \frac{q_k + q_{k+1}}{2},\quad & p_{k+\frac{1}{2}} &= \frac{p_k + p_{k+1}}{2},\\
{q^2}_{k+\frac{1}{2}} &= \frac{q^2_k + q^2_{k+1}}{2}, \quad & {p^2}_{k+\frac{1}{2}} &= \frac{p^2_k + p^2_{k+1}}{2}.	
\end{alignat}
Notice that $\overline{\nabla} V(z_k,z_{k+1})$ is symmetric with respect to $z_k$ and $z_{k+1}$, and it can be verified that it satisfies \eqref{dgcondition}.
A classical discrete gradient method can be constructed as follows,
\begin{align}\label{dgfornls}
  z_{k+1} = z_k + h J\overline{\nabla} H(z_k, z_{k+1}),
\end{align}
where
$$\overline{\nabla} H(z_k, z_{k+1}) = \begin{pmatrix}
-D_2 & 0 \\
0 & -D_2
\end{pmatrix}\frac{z_k + z_{k+1}}{2} + \overline{\nabla} V(z_k, z_{k+1}).$$
The method described by \eqref{dgfornls} is very similar to classical midpoint rule, the only difference is in $\overline{\nabla} V(z_k, z_{k+1})$, ${q^2}_{k+\frac{1}{2}}$ is used, while $(q_{k+\frac{1}{2}})^2$ is used in the midpoint rule, so \eqref{dgfornls} can be viewed as a modified midpoint rule.

\subsection{KdV equation}
Rewrite \eqref{kdv} as
\begin{equation*}
 u_t = \Big(-\frac{\partial}{\partial x}\Big)\Big(\frac{1}{2}u^2 + u_{xx}\Big),
\end{equation*}
then apply pseudospectral semi-discretization to obtain the following system,
\begin{align*}
  \dot{q} &= (-D_1)\Big(\frac{1}{2}q^2 + D_2q\Big) \\
          &= (-D_1)D_2q + (-D_1)\Big(\frac{1}{2}q^2\Big),
\end{align*}
which has the form of a semilinear Poisson system~\eqref{poisson},
\begin{equation*}
\dot{q} = J(Dq + \nabla V(q)) = J\nabla H(q),
\end{equation*}
where $J = -D_1$, $D = D_2$, $A = JD = -D_3$, $\nabla V(q) = \frac{1}{2}q^2$, $H(q) = \frac{1}{2}q^{\mathrm{T}}D_2q + \frac{1}{6}q^3$. The exponential midpoint rule for KdV
reads as follows,
\begin{equation*}
q_{k+1} = e^{-D_3h}q_k + h\cdot e^{-D_3\frac{h}{2}}f\Big(\frac{e^{-D_3\frac{h}{2}}q_k + e^{D_3\frac{h}{2}}q_{k+1}}{2}\Big),
\end{equation*}
and the energy preserving exponential integrator is given by,
\begin{equation*}
q_{k+1} = e^{-D_3h}q_k + \int_{0}^{h}e^{-D_3\tau}d\tau \cdot (-D_1)\overline{\nabla} V(q_k, q_{k+1}),
\end{equation*}
with discrete gradient $\overline{\nabla} V(q_k, q_{k+1})= \frac{1}{6}(q_k^2 + q_k\cdot q_{k+1} +q_{k+1}^2)$. A related
classical discrete gradient method can be constructed as follows,
\begin{align}\label{dgforkdv}
  \overline{\nabla} H(q_k, q_{k+1}) = (D_2)\frac{q_k + q_{k+1}}{2} + \overline{\nabla} V(q_k, q_{k+1}).
\end{align}

By Theorem~\ref{matrixtimesv}, in each iteration, the matrix function and vector product can be implemented as
$$e^{-D_3h}q = F^{-1}[e^{ik^3h}\hat{q}_k],$$ and
$$\Bigg(\int_{0}^{h}e^{-D_3\tau}d\tau\Bigg)q = \Big(\frac{e^{-D_3 h}-I}{-D_3}\Big)q = F^{-1}\Big[\frac{e^{ik^3h}-1}{ik^3}\hat{q}_k\Big].$$

\section{Numerical Experiments}\label{numerical_experiments}
\subsection{Nonlinear Schr\"{o}dinger equation}
In Table~\ref{table:maximum_timestep} below, $n$ denotes the number of nodes we discretize the spatial domain with, and the data in the table indicates the maximum timesteps for which the nonlinear solver converges. The first two columns correspond to midpoint rule, using fixed point iteration and Newton type iteration; the third column corresponds to the exponential midpoint rule, the fourth column is discrete gradient method \eqref{dgfornls}, and the last column is the energy-preserving exponential integrator, all implemented using fixed point iterations.

\begin{table}
\begin{tabular}{|c|c|c|c|c|c|}
\hline
& \multicolumn{2}{c|}{midpoint}               & midpoint exp     &discrete gradient       &energy exp\\
\hline $n$   & fixed point       & Newton & fixed point      & fixed point            & fixed point \\
\hline 11    & 0.02              &        0.1       &    0.1           &0.02                    & 0.1 \\
\hline 21    & 0.01              &        0.1       &   0.08           &0.01                    & 0.1 \\
\hline 41    & $4\times 10^{-3}$ &        0.1       & 0.06             & $4\times 10^{-3}$      & 0.1 \\
\hline 61    & $2\times 10^{-3}$ & 0.1              & 0.04             & $2\times 10^{-3}$      & 0.1 \\
\hline 81 & $10^{-3}$         & 0.1 & 0.04                    & $10^{-3}$                 & 0.1 \\
\hline 121 & $5\times 10^{-4}$ & 0.1 & 0.04                  & $5\times 10^{-4}$                 & 0.1 \\
\hline 161 & $2\times 10^{-4}$ & 0.1 & 0.01                 & $2\times 10^{-4}$                   & 0.1 \\
\hline 201 & $1\times 10^{-4}$ & 0.1 & $8\times 10^{-3}$     & $1\times 10^{-4}$                  & 0.1 \\
\hline 401 & $4\times 10^{-5}$ & 0.1 & $5\times 10^{-3}$    & $4\times 10^{-5}$                   & 0.1 \\
\hline
\end{tabular}
\caption{Maximum timestep for nonlinear solver to converge as a function of numerical integrator, nonlinear solver, and spatial resolution.}
\label{table:maximum_timestep}
\end{table}

We observe that when the midpoint rule is implemented using fixed point iterations, there is a significant decay in the allowable timestep as the spatial resolution is increased, whereas the Newton type iteration allows a relatively large timestep that is independent of the spatial resolution. In contrast, the midpoint exponential method exhibits a slower rate of decrease in allowable timestep when using fixed point iterations. When using fixed point iterations, the discrete gradient method, which is an energy preserving method, the allowable timestep behaves similarly to the midpoint rule, and in contrast, the energy preserving exponential integrator has an allowable timestep that is independent of the spatial resolution.

\begin{figure}[H]
\includegraphics[scale=0.75]{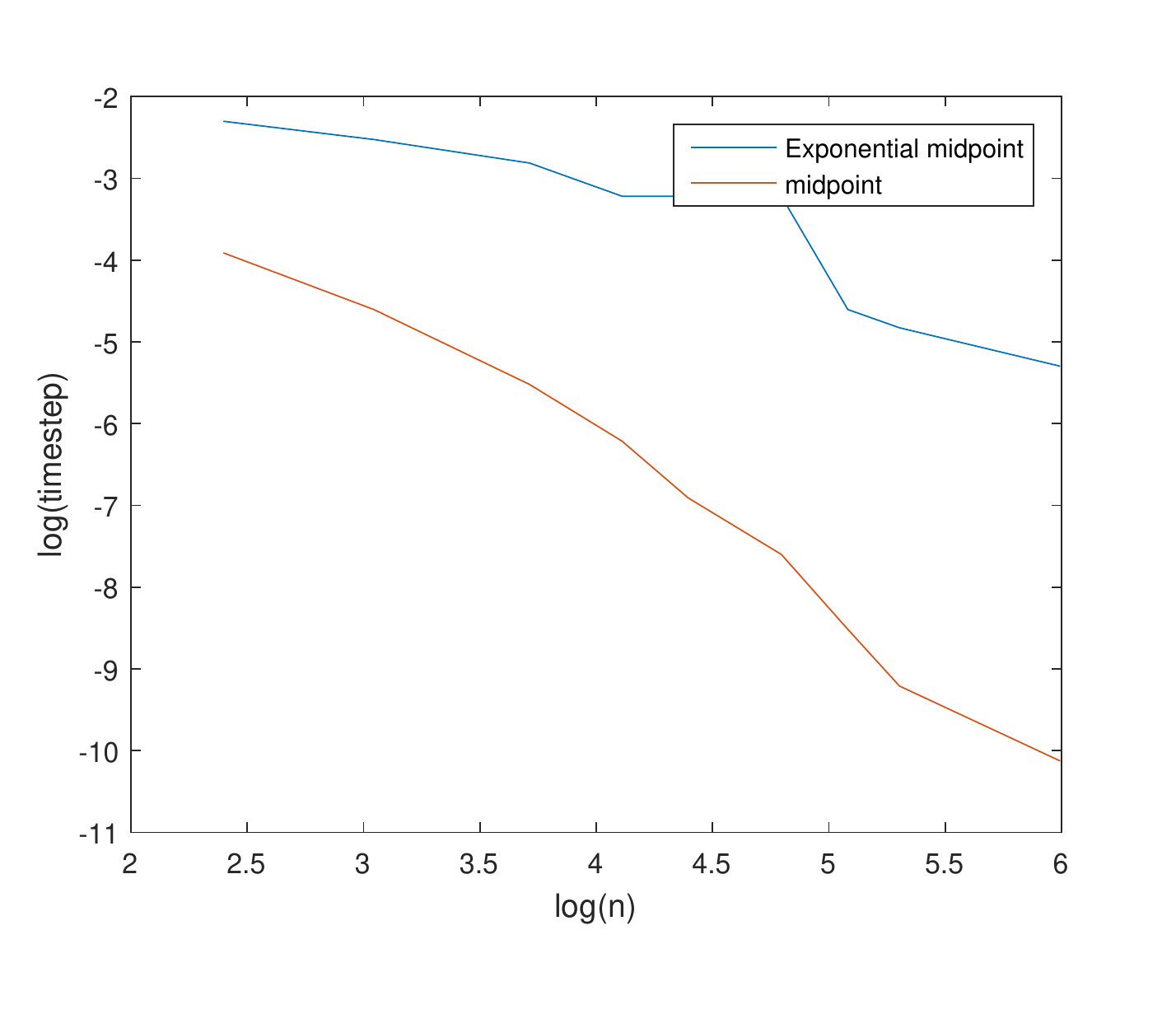}
\caption{Maximum timestep for which fixed point iterations converge as a function of the spatial resolution.\label{fig:timestep_vs_resolution}}
\end{figure}

In Figure~\ref{fig:timestep_vs_resolution}, we observe that the allowable timestep when using fixed point iteration scale like $n^{-2}$ for the classical midpoint rule, and $n^{-1}$ for the midpoint exponential rule.

\begin{figure}[H]
	\begin{subfigure}[b]{0.45\textwidth}
		\includegraphics[scale=0.4]{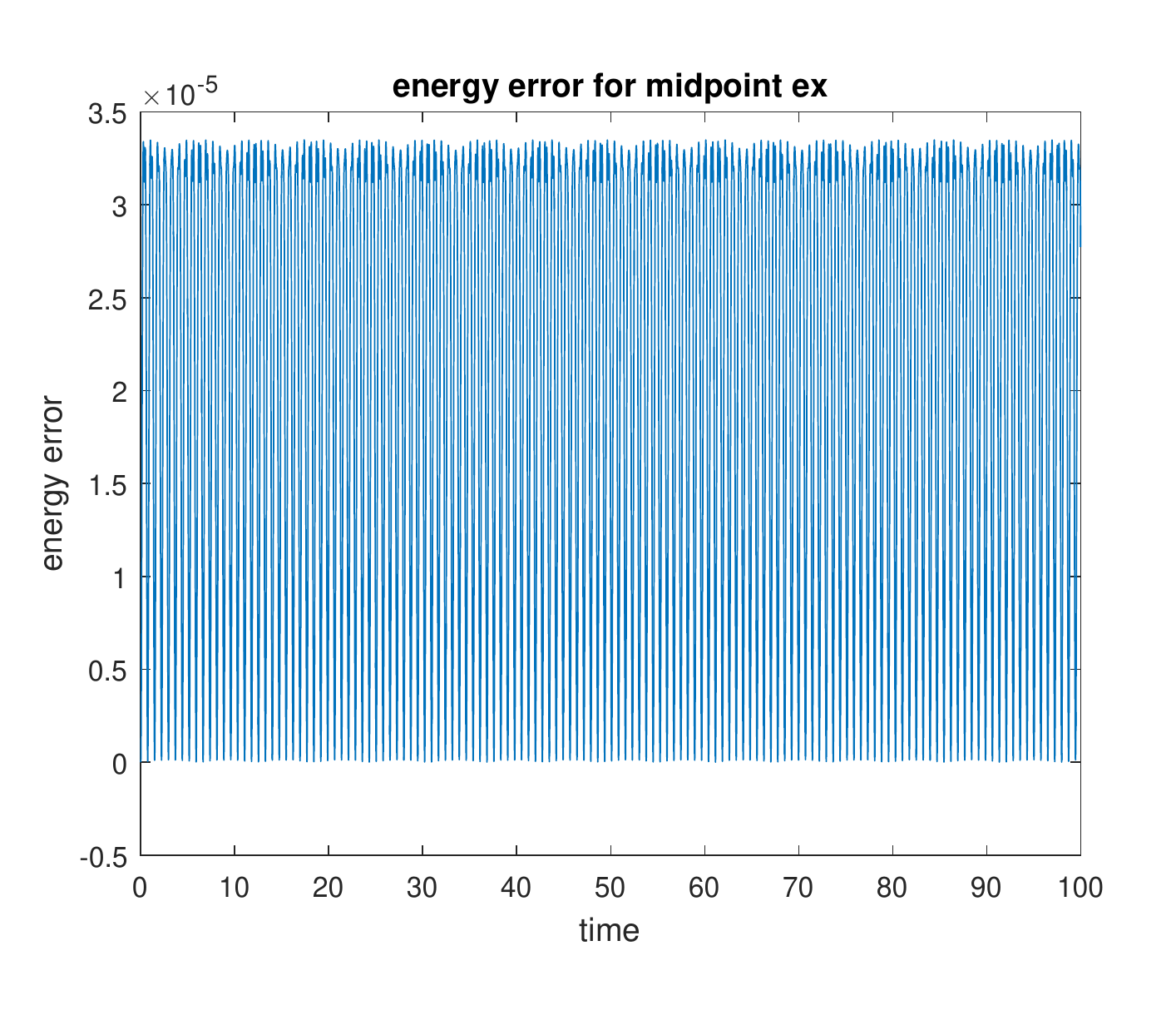}
		\caption{Energy error}
	\end{subfigure}
	\begin{subfigure}[b]{0.45\textwidth}
		\includegraphics[scale=0.4]{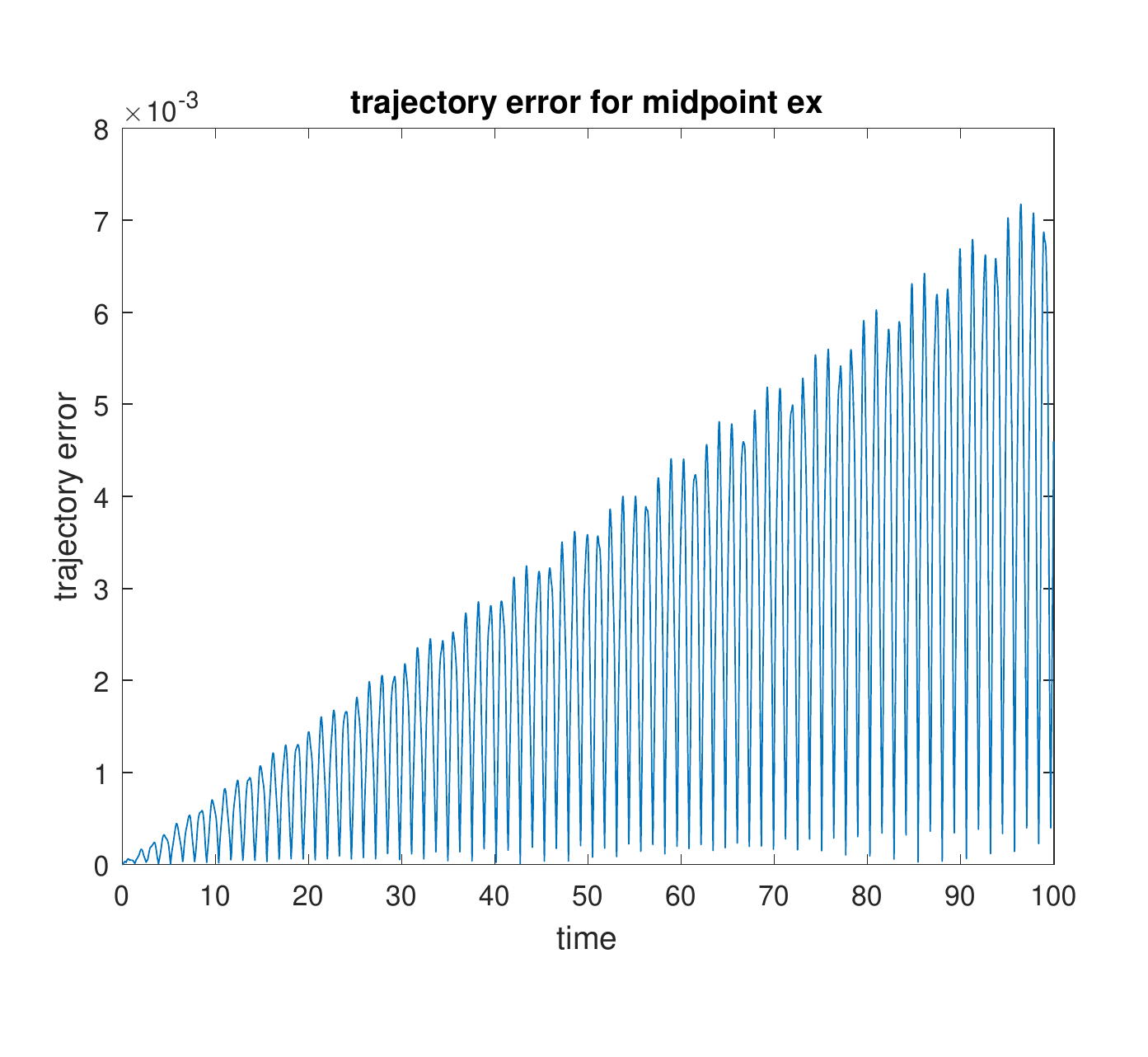}
		\caption{Trajectory error}
\end{subfigure}
\caption{Error plots for the exponential midpoint rule, $n = 161$, $timestep = 0.01$.\label{fig:exp_midpoint}}
\end{figure}

As shown in Figure~\ref{fig:exp_midpoint}, the exponential midpoint rule exhibits an energy error that remains small and bounded, which is consistent with it being a symplectic integrator, and the trajectory error grows linearly.
\newline

\begin{figure}[H]
	\begin{subfigure}[b]{0.45\textwidth}
		\includegraphics[scale=0.4]{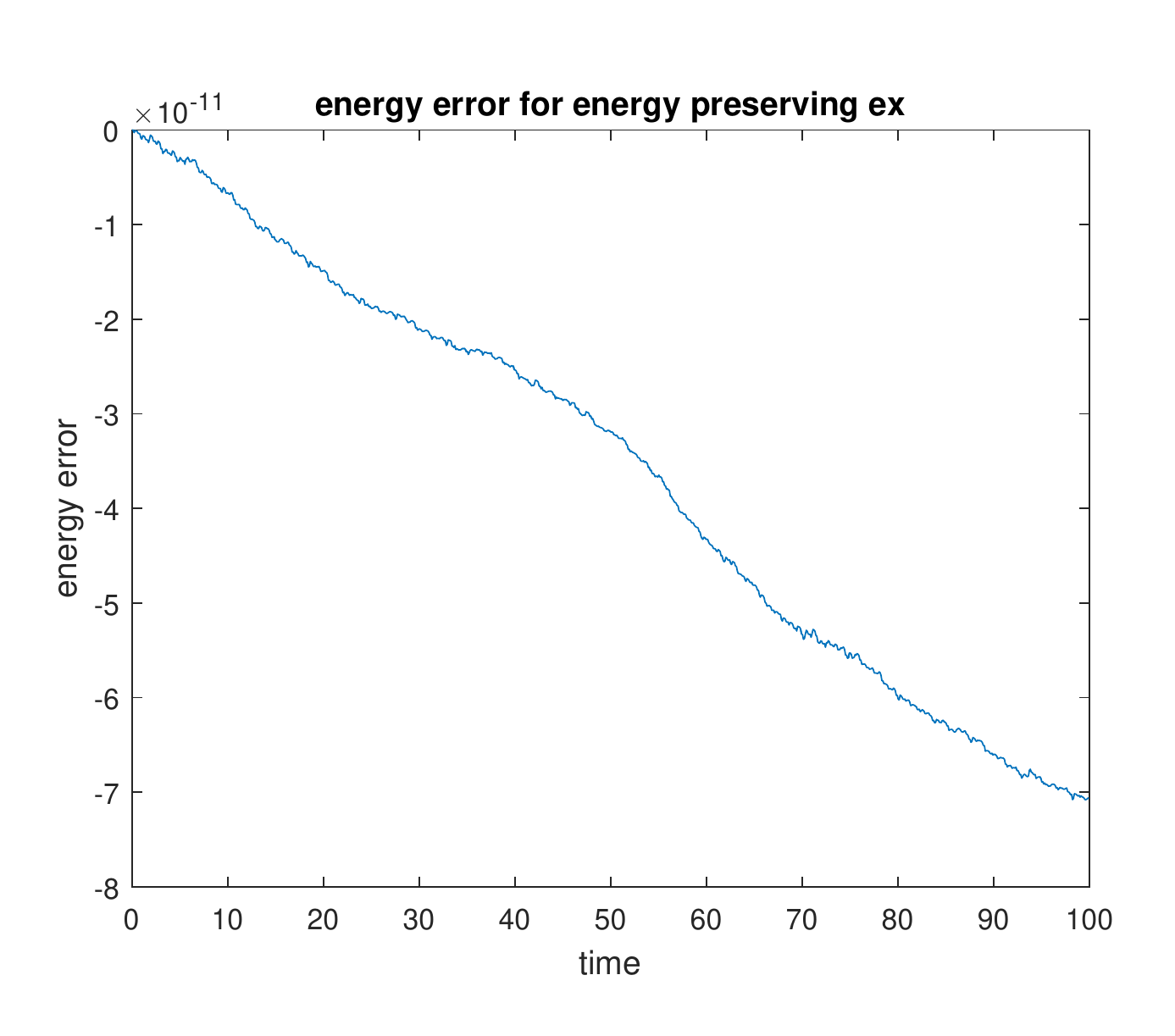}
		\caption{Energy error}
	\end{subfigure}
	\begin{subfigure}[b]{0.45\textwidth}
		\includegraphics[scale=0.4]{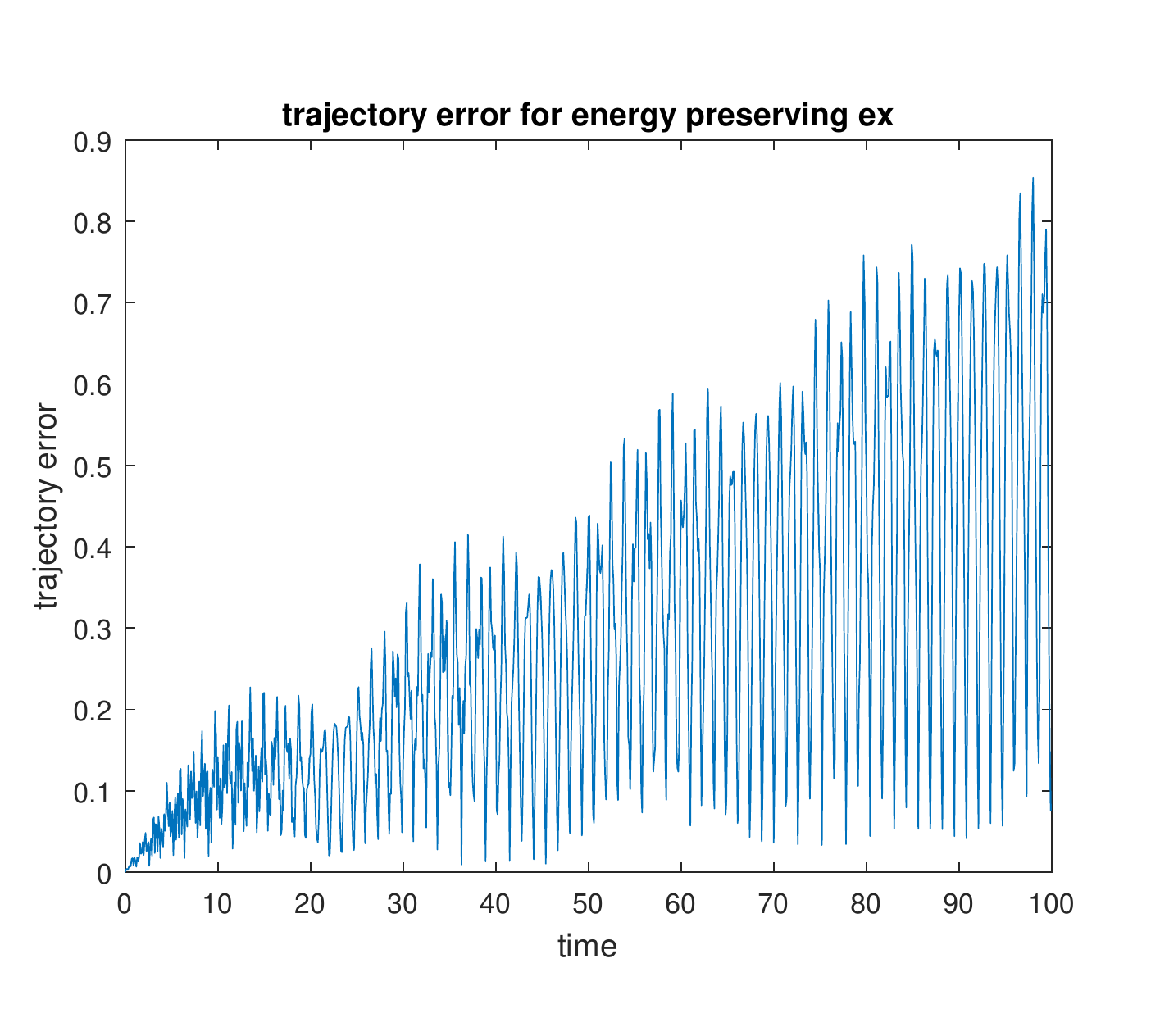}
		\caption{Trajectory error}
\end{subfigure}
\caption{Error plots for the energy preserving exponential integrator, $n = 161$, $timestep = 0.1$.\label{fig:energy_exp}}
\end{figure}

The energy preserving exponential integrator is designed to preserve energy exactly, so even when the timestep is
$0.1$, we see in Figure~\ref{fig:energy_exp} that the energy is still preserved approximately to within machine error.
\newline

We now consider the diagonally implicit symplectic exponential (DISEX) integrator with six stages:
\begin{align*}
    b_1    &= 0.5080048194000274     & b_2  &= 1.360107162294827  &
      b_3  &= 2.019293359181722               \\
    b_4   &= 0.5685658926458250     &b_5 &= -1.459852049586439 &
      b_6 &= -1.996119183935963.                \\
\end{align*}

The energy and trajectory error is shown in Figure~\ref{fig:disex}. Observe that the energy error is small and bounded, as expected of a symplectic integrator, and the trajectory error is small as well, as expected of a higher-order numerical integrator.

\begin{figure}[H]
	\begin{subfigure}[b]{0.45\textwidth}
		\includegraphics[scale=0.4]{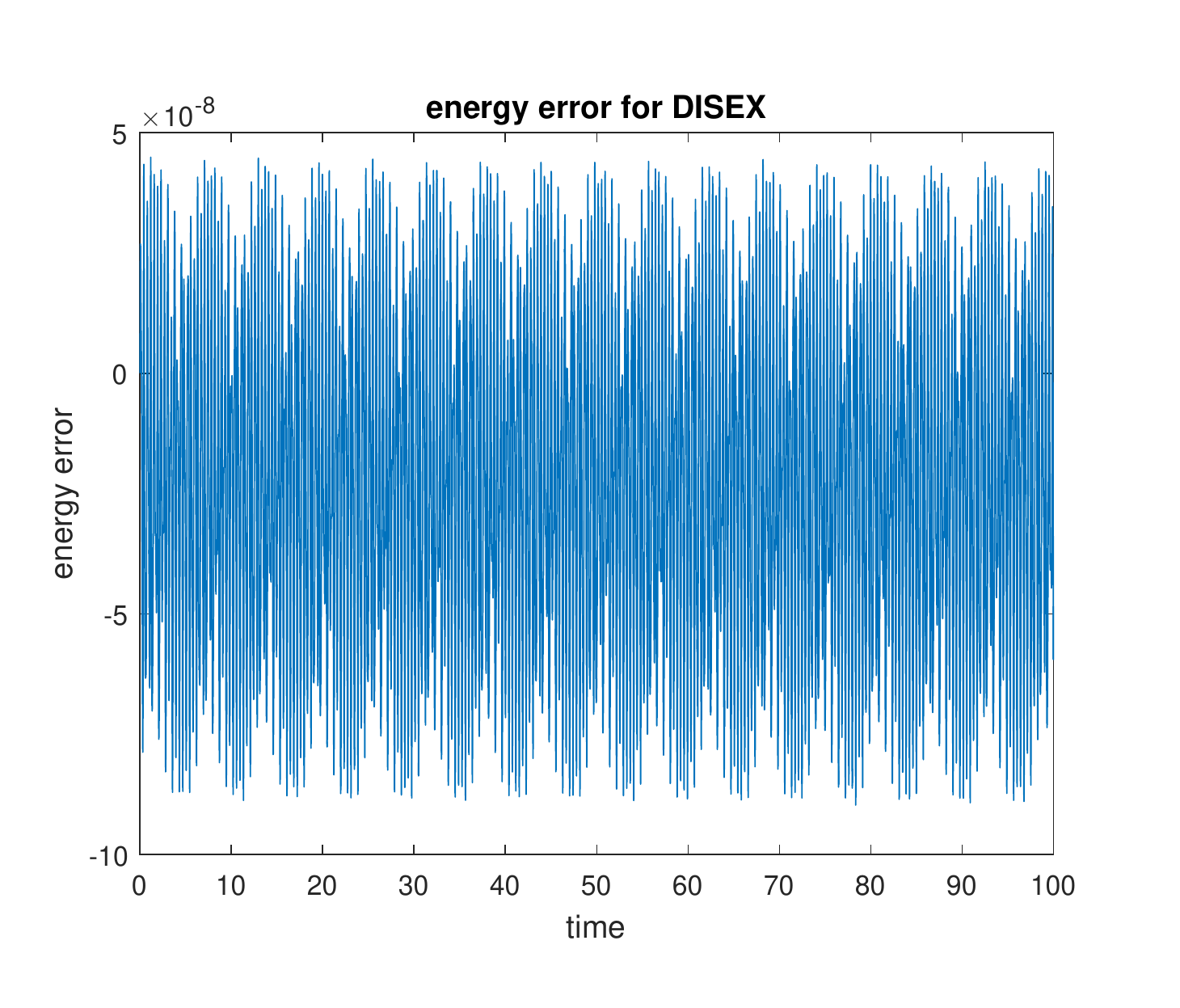}
		\caption{Energy error}
	\end{subfigure}
	\begin{subfigure}[b]{0.45\textwidth}
		\includegraphics[scale=0.4]{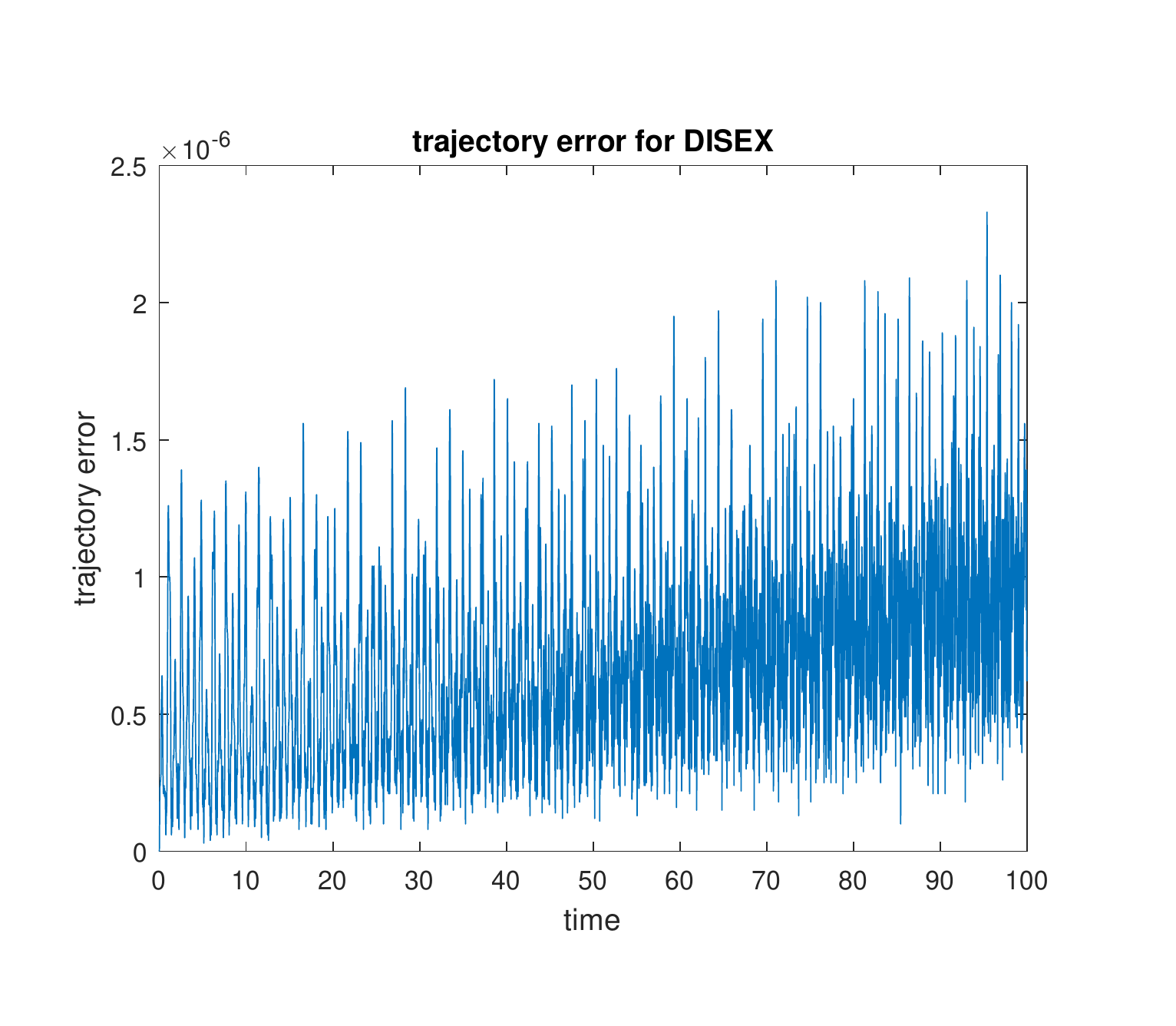}
		\caption{Trajectory error}
\end{subfigure}
\caption{Error plots for the 6 stage diagonally implicit symplectic exponential integrator (DISEX), $n = 161$, $timestep = 0.01$.\label{fig:disex}}
\end{figure}

\subsection{KdV} We simulate the KdV equation,
\[u_t + uu_x + \nu u_{xxx} = 0,\]
where $\nu = 5\times 5^{-4}$. As before, we explore how the maximum timestep for which the fixed point iteration converges depends on the choice of numerical integrator, and the spatial resolution of the semi-discretization. In Table~\ref{table:maximum_timestep_kdv}, $n$ is the number of nodes we use to discretize the spatial domain, the data in the table indicates the maximum timestep for which fixed point iterations converge. For the midpoint rule, the timestep decreases like $n^{-3}$ for fixed point iterations, and a comparable timestep is required for Newton iterations which is thereby too costly to implement. The classical discrete gradient method for the KdV equation is given by \eqref{dgforkdv}, and it exhibits the same timestep restrictions as the midpoint rule. In contrast, both the exponential midpoint and energy preserving exponential integrator allow rather large timesteps that are independent of the spatial resolution.

\begin{table}[h]
\begin{tabular}{|c|c|c|c|c|}
\hline $n$    & midpoint               & midpoint exp            &discrete gradient                       & energy exp \\
\hline 401  & $4\times 10^{-4}$   &    $8\times 10^{-4}$     &$4\times 10^{-4}$                         & 0.005 \\
\hline 601  & $1\times 10^{-4}$   &   $6\times 10^{-4}$      &$1\times 10^{-4}$                          & 0.005 \\
\hline 801  & $6\times 10^{-5}$   & $6\times 10^{-4}$       &  $5\times 10^{-5}$                        & 0.005 \\
\hline 1001 & $3\times 10^{-5}$   & $6\times 10^{-4}$        & $3\times 10^{-5}$                        & 0.005 \\
\hline 1201 & $1\times 10^{-5}$   & $6\times 10^{-4}$        &$1\times 10^{-5}$                         & 0.005 \\
\hline 1401 & $1\times 10^{-5}$   & $6\times 10^{-4}$        & $1\times 10^{-5}$                        & 0.005 \\
\hline
\end{tabular}
\caption{Maximum timestep for fixed point iteration to converge as a function of numerical integrator, and spatial resolution.}
\label{table:maximum_timestep_kdv}
\end{table}

In Figure~\ref{fig:exp_midpoint_kdv}, we observe that the exponential midpoint rule has an energy error that is small and bounded, as is typical for a symplectic integrator, and the trajectory error grows linearly. In Figure~\ref{fig:energy_exp_kdv}, the energy preserving exponential integrator has an energy that is preserved to within machine precision, and the trajectory error grows linearly.

\begin{figure}[H]
	\begin{subfigure}[b]{0.45\textwidth}
		\includegraphics[scale=0.4]{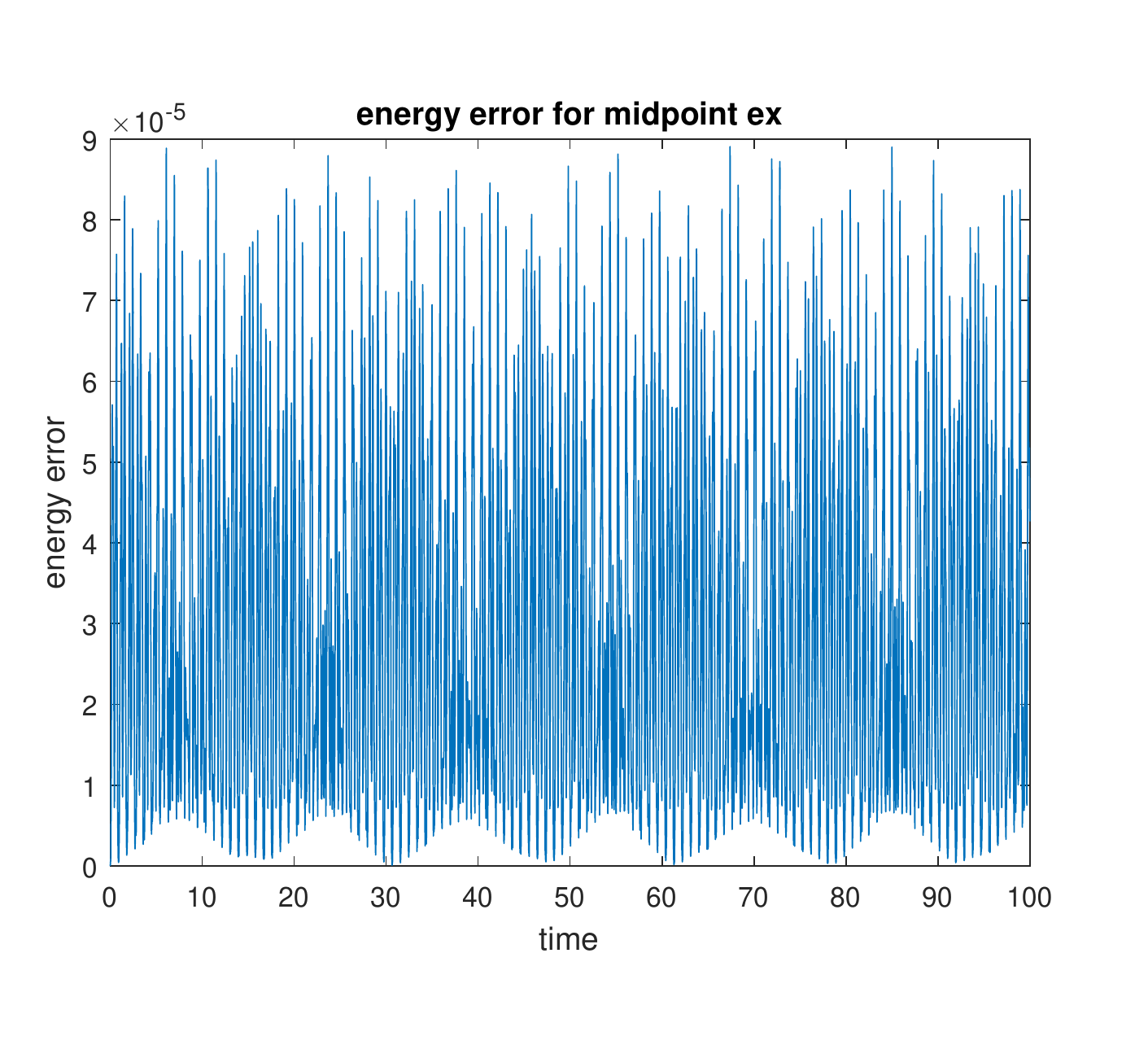}
		\caption{Energy error}
	\end{subfigure}
	\begin{subfigure}[b]{0.45\textwidth}
		\includegraphics[scale=0.4]{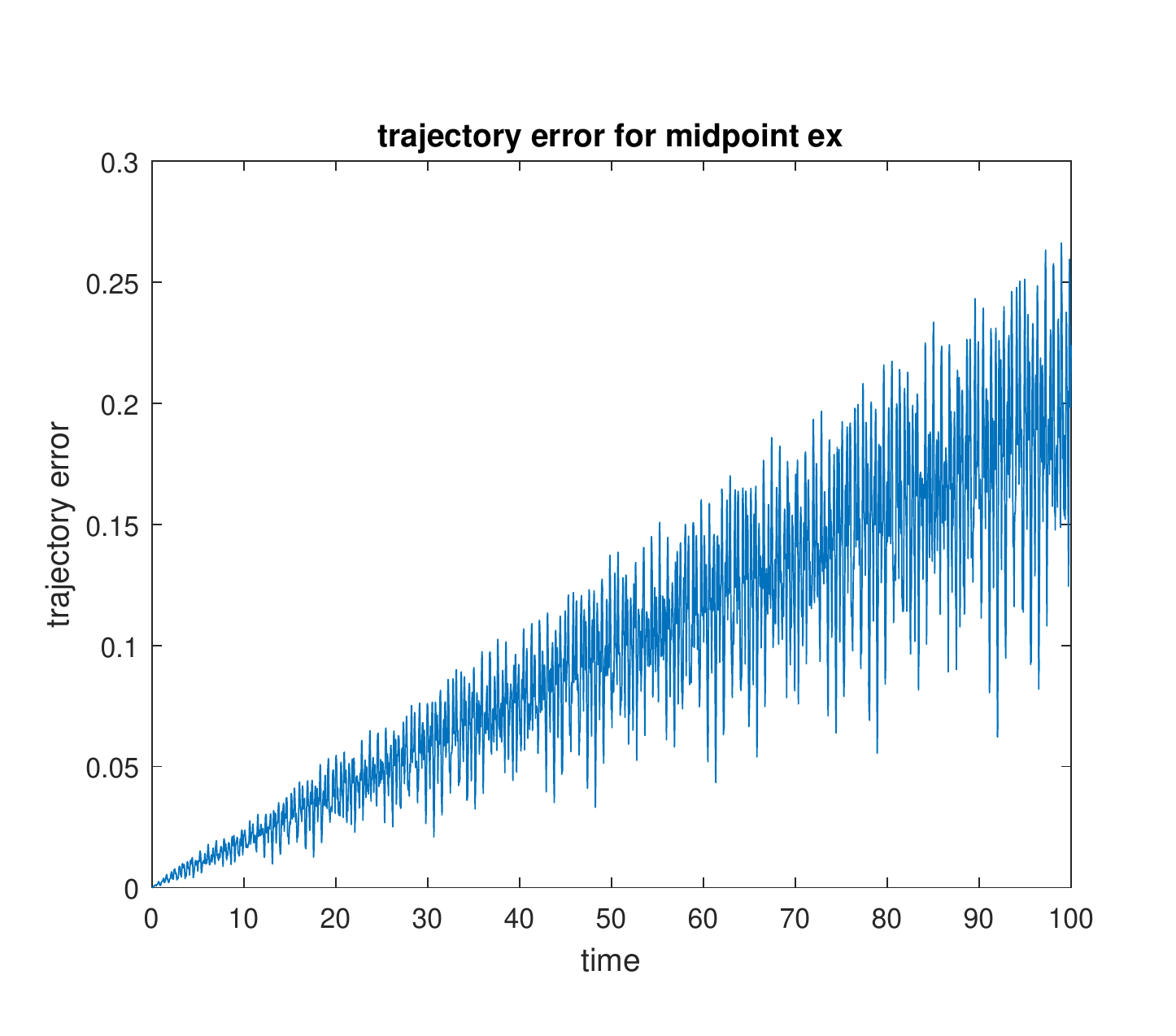}
		\caption{Trajectory error}
\end{subfigure}
\caption{Error plots for the exponential midpoint rule, $n = 401$, $timestep = 5\times 10^{-4}$.\label{fig:exp_midpoint_kdv}}
\end{figure}

\begin{figure}[H]
	\begin{subfigure}[b]{0.45\textwidth}
		\includegraphics[scale=0.4]{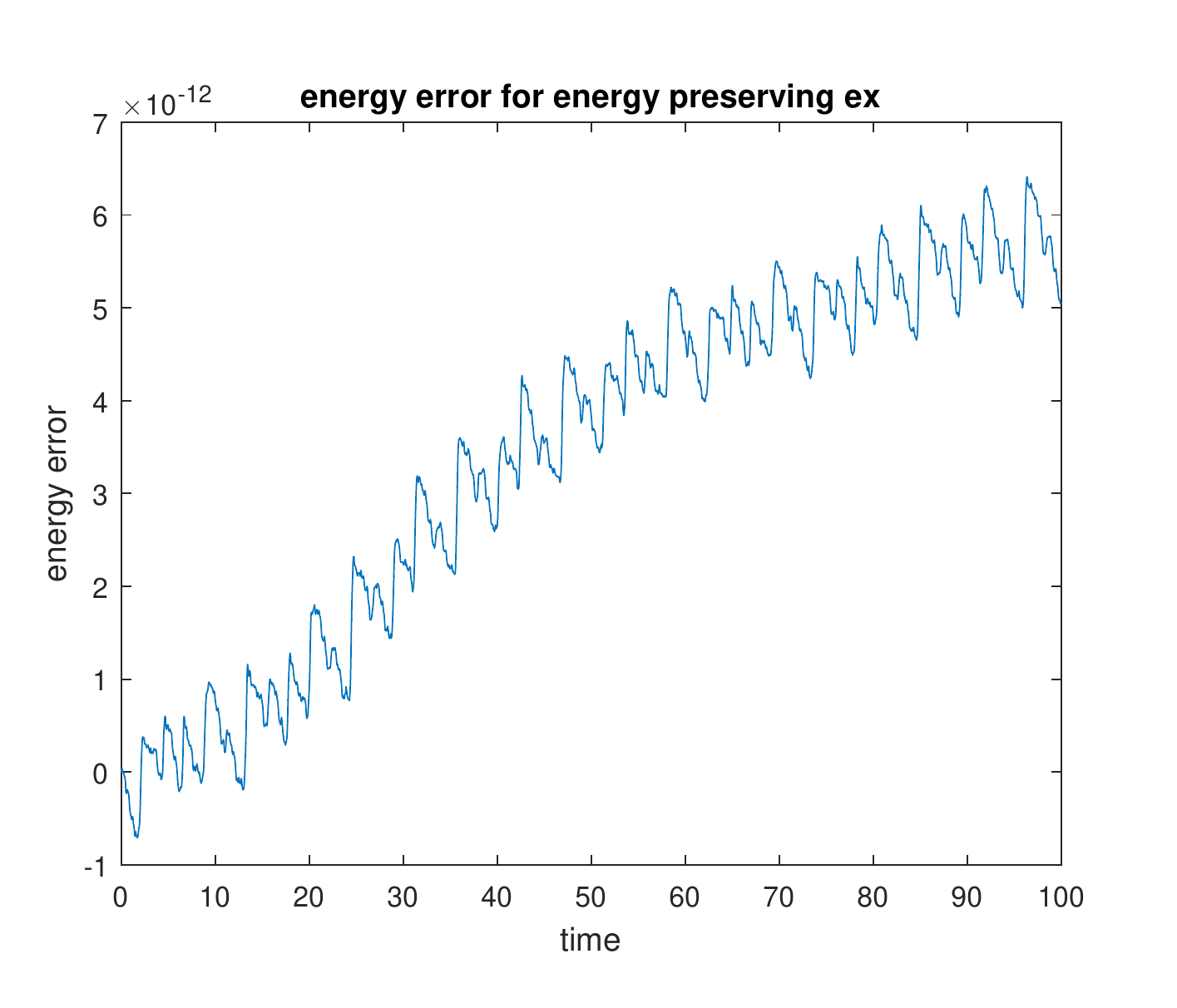}
		\caption{Energy error}
	\end{subfigure}
	\begin{subfigure}[b]{0.45\textwidth}
		\includegraphics[scale=0.4]{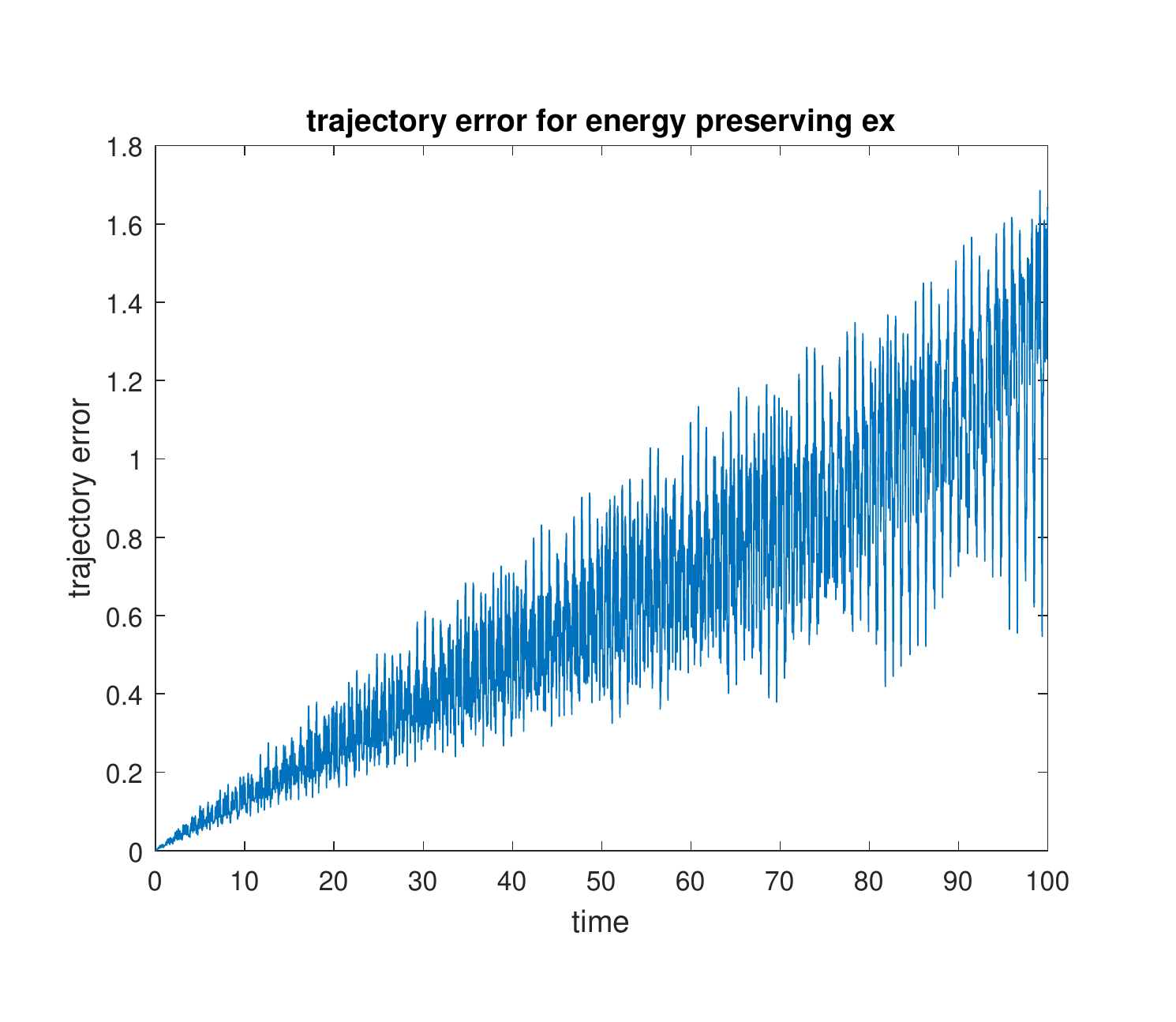}
		\caption{Trajectory error}
\end{subfigure}
\caption{Error plots for the energy preserving exponential integrator, $n = 401$, $timestep = 0.005$.\label{fig:energy_exp_kdv}}
\end{figure}

\section{Summary}
For semilinear Poisson system, we have developed two types of geometric exponential integrators, one that preserves the Poisson structure, and the other preserves the energy. They allow fixed point iteration methods to be used for significantly larger timesteps as compared to non-exponential integrators, such as the classical midpoint rule and the standard discrete gradient method, which require the use of Newton type iterations to converge with comparably large timesteps. This results in substantial computational savings, particularly when applied to the simulation of semi-discretized Hamiltonian PDEs. They also exhibit long time stability and conservation of the first integrals. We notice that in practice, energy preserving exponential integrators are more stable and allow for larger timesteps than symplectic exponential integrators. In future work, we will develop higher-order energy preserving exponential integrators.

\section*{Acknowledgements}
This research has been supported in part by NSF under grants DMS-1010687, CMMI-1029445, DMS-1065972, CMMI-1334759, DMS-1411792, DMS-1345013.

\clearpage
\bibliography{exponential}
\bibliographystyle{plainnat}
\end{document}